\documentclass[conference]{IEEEtran}
\IEEEoverridecommandlockouts
\usepackage{cite}
\usepackage{amsmath,amssymb,amsfonts}
\usepackage{algorithmic}
\usepackage{graphicx}
\usepackage{textcomp}
\def\BibTeX{{\rm B\kern-.05em{\sc i\kern-.025em b}\kern-.08em
    T\kern-.1667em\lower.7ex\hbox{E}\kern-.125emX}}

\usepackage{amsthm} 

\usepackage{hyperref}
\hypersetup{colorlinks, citecolor=blue, filecolor=black, linkcolor=blue, urlcolor=blue}

\input{mystyle.sty}

\begin{document}

\title{Sparse Multiple Kernel Learning: \\ Support  Identification via Mirror Stratifiability
\thanks{
The work of G. Garrigos has been supported by the European Research Council (ERC project NORIA).
L. Rosasco is funded by the Air Force project FA9550-17-1-0390 (European Office of Aerospace Research and Development) and by the FIRB project RBFR12M3AC (Italian Ministry of Education, University and Research).
Both L. Rosasco and S. Villa are funded by the RISE project NoMADS - DLV-777826.
}
}

\author{\IEEEauthorblockN{Guillaume Garrigos}
\IEEEauthorblockA{\textit{\'Ecole Normale Sup\'erieure, CNRS} \\
Paris, France \\
guillaume.garrigos@ens.fr}
\and
\IEEEauthorblockN{Lorenzo Rosasco}
\IEEEauthorblockA{
\textit{Universit\`a degli Studi di Genova,}\\
\textit{Massachusetts Institute of Technology,}\\
\textit{Istituto Italiano di Tecnologia,}\\
Genova, Italy \\
lrosasco@mit.edu}
\and
\IEEEauthorblockN{Silvia Villa}
\IEEEauthorblockA{\textit{Dipartamento di Matematica}\\
\textit{Politecnico di Milano} \\
Milano, Italy \\
silvia.villa@polimi.it}
}

\maketitle

\begin{abstract}
In statistical machine learning, kernel methods allow to  consider infinite dimensional feature spaces 
with a computational cost that only depends on the number of observations. This is usually done by
solving an optimization problem depending on a data fit term and a suitable regularizer. In this paper
we consider feature maps which are the concatenation of a fixed, possibly large, set of simpler feature maps. The
penalty is a sparsity inducing one, promoting solutions depending only on a small subset of the features.
The group lasso problem is a special case of this more general setting.
We show that one of the most popular optimization algorithms to solve the regularized objective
function,  the forward-backward splitting method, allows to perform feature selection in a stable manner. 
In particular, we prove that the set of relevant features is identified by the algorithm after a finite number of iterations if a suitable qualification condition holds. 
Our analysis rely on the notions of stratification and mirror stratifiability.
\end{abstract}


\section{Introduction}

Kernel methods provide a practical and elegant way to derive nonparametric approaches to supervised learning \cite{SteiChri08}. 
In this context, the goal is to determine a function $f_*\colon \mathbb{R}^d\to \mathbb{R}$ observing a finite training set of
data points $\{(x_1,y_1),\ldots,(x_m,y_m)\}\subset \mathbb{R}^d\times \mathbb{R}$.  One of the most popular approaches to find estimates of $f_*$ relies on the solution of an empirical  minimization problem 
\begin{equation}
\label{e:RERM}
\min_{f\in\HH} \ \ \lambda R(f)+\frac{1}{2m}\sum\limits_{i=1}^{m} (f(x_i) - y_i )^2, \quad\lambda>0,
\end{equation}
where
$(\mathcal{F},\langle\cdot,\cdot\rangle)$ is an appropriate  Hilbert space  of real valued functions  on $\mathbb{R}^d$ and  
$R\colon \HH\to\mathbb{R}$ a  regularization functional.
The choice of the  space $\HH$ and the regularization $R$ is crucial, since it  allows  to incorporate  a priori information on the problem \cite{Vap98}.  In kernel methods, $\HH$ 
is realized through a (nonlinear) feature map $\Phi\colon \mathbb{R}^d\to H$, for  a separable Hilbert space $H$, in the sense that
\[
(\forall f\in \HH)(\forall x\in\mathbb{R}^d)\quad f(x)=\langle w, \Phi(x)\rangle_{H}, 
\]
for some $w\in H$. A positive definite kernel can then be defined by setting $k(x,x')=\langle\Phi(x),\Phi(x')\rangle_{H}$. In this paper, we focus on a class of structured feature spaces, that can be written 
as direct sums of a finite number of  spaces, involving thus different feature maps. 
This choice is known as multiple kernel learning \cite{BacLanJor04,MicPon05,MicPon07}.

\smallskip

\noindent \textbf{Sparse Multiple Kernel Learning}.
We consider $H_1,\ldots,H_G$  a finite number of separable Hilbert spaces and, for every $i\in[\![ 1,...,G ]\!] $, we let $\Phi_i\colon \mathbb{R}^d\to H_i$ to be a map. 
Then, we consider the space of functions defined by the concatenated feature map 
\[
\Phi\colon \mathbb{R}^d\to H_1\times\ldots\times H_G, \quad \Phi=(\Phi_1,\ldots,\Phi_G).
\]
The supervised learning problem in $\HH$ is called multiple kernel learning. 
In the context of multiple kernel learning,  $f_*$ is assumed to have a sparse representation in the feature space $H := H_1 \times \ldots \times H_G$, namely 
\[
(\forall x\in \mathbb{R}^d)\quad f_*(x)=\sum_{g=1}^G \langle  (w_*)_g,\Phi_g(x)  \rangle_{H_g},
\]
 with  $(w_*)_g=0$ in $H_g$ for several $g$'s. To approximate the function $f_*$, the idea is to devise a regularization term inducing sparsity in the sum 
 representations \cite{BacLanJor04,MicPon05,MicPon07}. This leads to the minimization problem with respect to $w =(w_1,\ldots,w_G) \in H$

\begin{equation}\label{P:ERM}
\underset{w \in H}{\argmin } \  \lambda \sum\limits_{g=1}^{G} \Vert w_g \Vert_{H_g} + \frac{1}{2m} \sum\limits_{i=1}^{m} (\sum_{g=1}^G \langle w_g , \Phi_g(x_i) \rangle_{H_g} - y_i )^2. 
\tag{MKL}
\end{equation}
The MKL  approach has several possible motivations. First, it allows the use of richer hypotheses spaces, and therefore allows in principle for better solutions to be found. Indeed, we can think of it as a way to combine different feature maps/kernels. Moreover, multiple kernel learning 
can also be used  as an alternative to model selection. By creating a concatenation of feature maps, model selection is replaced by an aggregation procedure.  Finally, multiple kernel 
learning is useful for data fusion. For example, different features/kernels could capture different characteristics of the data, and a better description is obtained by combining them. In the regime where we consider a large number of feature maps, the identification of the active features is crucial for good results. The special  case of MKL, where each feature map is finite dimensional is equivalent to so called {\em group lasso}.

\begin{example}[Group lasso] 
Partition $\mathbb{R}^d$ into non overlapping groups of variables in $\mathbb{R}^{d_i}$ for  $i\in[\![1,\ldots,G]\!]$. Set $H_i=\mathbb{R}^{d_i}$ and let $\Phi_i\colon\mathbb{R}^d\to\mathbb{R}^{d_i}$ be the canonical projection.  Then $H=\mathbb{R}^{d_1}\times\ldots\times \mathbb{R}^{d_G}$ and problem~\eqref{P:ERM} becomes
\begin{equation}
\label{e:GL}
\tag{GL}
\underset{w\in\mathbb{R}^{d}}{\argmin } \  \lambda \sum\limits_{g=1}^{G} \Vert w_g \Vert_{\mathbb{R}^{d_g}} + \frac{1}{2m} \sum\limits_{i=1}^{m} (\langle w ,x_i \rangle - y_i )^2.
\end{equation}
This is a well studied problem, and is called group lasso \cite{YuaLin06}. It is of interest when the vector $w_*$ has a structured sparsity pattern. 
More precisely, the considered penalty induces block-wise sparse  vectors, where all the
components belonging to a  given group are zero.
\end{example}

\smallskip

\noindent \textbf{The Iterative Kernel Thresholding Algorithm}. 
From a computational point of view,  as observed in \cite{MosRosSan10}, problem \eqref{P:ERM} can be solved using proximal splitting methods \cite{BauComV2}: the objective function structure is amenable for the forward
backward splitting algorithm,  since it is the sum of a differentiable term (the square loss) with Lipschitz continuous gradient, and  a convex continuous regularizer.   
Though this algorithm is defined on a possibly infinite-dimensional space 
$H$, a generalization of the representer theorem allows to compute its iterates as a finite linear combination of the features \cite{MosRosSan10}, as discussed next.

For every $g\in[\![1,\ldots,G]\!]$, let
\begin{align*}
X_g\colon& H_g\to\mathbb{R}^m, \quad
w_g\mapsto (\langle \Phi_g(x_i),w_g \rangle)_{i=1}^m,
\end{align*}
and 
\begin{equation*}
K_g := X_g X_g^* = (\langle \Phi(x_i), \Phi(x_{i'})\rangle )_{ii'}, \quad { K}=\left[K_1,\dots,K_G\right].
\end{equation*}
Then, the Iterative Kernel Thresholding Algorithm can be written as
\begin{equation}
\label{E:algo}
 \tag{$\text{IKTA}$}
\begin{array}{|l}
\alpha^0\in\mathbb{R}^{m\times G}, \ \tau \in ]0,2 \Vert { K} \Vert^{-1}[, \\
\alpha_g^{n+1} = \widehat \G_{\lambda \tau,g} \left( \alpha_g^n - \tau (K \alpha^n -{y}) \right) ,\\
w^{n+1}_g = X_g^* \alpha_g^{n+1}
,
\end{array}
\end{equation}
where $\widehat \G_{\lambda,g}$ denotes a group-thresholding operator defined by:
\begin{equation*}
\widehat \G_{\lambda,g} (\alpha_g) := 
\begin{cases}
 0 & \text{ if } \Vert X_g^* \alpha_g \Vert_{H_g} \leq \lambda , \\
 \left(1 - \frac{\lambda}{\Vert X_g^* \alpha_g \Vert_{H_g}} \right) \alpha_g & \text{ otherwise.}
\end{cases}
\end{equation*}
In the rest of the paper we show that the proposed algorithm produces sparse iterates, 
and more notably is able to identify the relevant features of the solutions of the regularized empirical risk functional if
a suitable qualification condition holds.  
This is a crucial and desirable property in sparse regularization, since the exact solutions of the 
regularized problems will be never computed in closed 
form and the sparsity pattern gives key information on the sought model.  In case the qualification condition does not hold, we 
show that the support of the iterates, though in general bigger than the ideal one, can still be controlled.

\section{Main results}

As a preliminary result, we state the asymptotic convergence of IGTA towards a solution of \eqref{P:ERM}.

\begin{proposition}\label{P:CV of the algorithm}
Let $(w^n)_{\nin}$ be a sequence generated by \ref{E:algo}.
Then there exists $\bar w \in H$, a solution of \eqref{P:ERM}, such that $\Vert w^n - \bar w \Vert \rightarrow 0$ as $n \to + \infty$.
\end{proposition}

In our first main result, we look at the \textit{exact support recovery} property of  \ref{E:algo}.
We recall that the (group) support of a vector $w \in H = H_1 \times \ldots H_G$ is given by
\begin{equation*}
\supp(w) := \{ g \in [\![ 1 \ldots G ]\!] : w_g \neq 0 \}.
\end{equation*}
We show that, after a finite number of iterations, the algorithm identifies exactly the support of $\bar w$, provided the following qualification condition is verified:
\begin{equation}\label{QC}
\max\limits_{g \notin \supp (\bar w)} \  \frac{1}{\lambda} \Vert X_g^*(X \bar w - y ) \Vert_{H_g} < 1, \tag{QC}
\end{equation}
where $X : H \rightarrow \R^m$ is given by $Xw = \sum_{g=1}^{G} X_g w_g$.
This condition is also known as the ``irrepresentabilty condition''  \cite{zhao2006model,bach2008group}, or the ``strict sparsity pattern''  \cite{BreLor08}.

\begin{theorem}[Exact support recovery]\label{T:exact recovery with QC}
Let $(w^n)_{\nin}$ be a sequence generated by \ref{E:algo}, converging to $\bar w$.
If \eqref{QC} holds, then there exists $N \in \N$ such that
\begin{equation*}
(\forall n \geq N) \quad \supp(w^n) = \supp(\bar w).
\end{equation*}
\end{theorem}

The condition \eqref{QC} is necessary for Theorem \ref{T:exact recovery with QC}, in the sense that the support of $\bar w$ \textit{cannot be identified in general} after a finite number of iterations.
Without \eqref{QC}, it could be that the algorithm generates sequences identifying a strictly larger support. 
This is illustrated in the following example:
\begin{example}
Let $G=1$, $H_1=\R$, $X_1=1$, $y_1=1$, $\lambda =1$.
Then the corresponding problem \eqref{P:ERM} has a unique solution, $\bar w=0$, for which \eqref{QC} is clearly not satisfied.
Take now  $\alpha^0=w^0 \in ]0,+\infty[$, $\tau \in ]0,1[$, and consider the sequence $(w^n)_\nin$ generated by \ref{E:algo}.
It is not difficult to show that this sequence verifies $w^{n+1} = (1-\tau) w^n$, and therefore will never satisfy $\supp(w^n) = \supp(\bar w)$.
\end{example}

In our second main result, we tackle the case in which \eqref{QC} is not verified, and show that the support of the sequence can nevertheless  be controlled.
For this, one needs to look at the \textit{extended support} of the solution, which corresponds to the set of  active constraints of the dual problem of \eqref{P:ERM}:
\begin{equation*}
\esupp(\bar w) := \{ g \in [\![ 1,...,G ]\!] : \Vert X_g^*(X \bar w - y ) \Vert_{H_g} = \lambda \}.
\end{equation*}

\begin{theorem}[Extended support control]\label{T:identification with no QC}
Let $(w^n)_{\nin}$ be a sequence generated by \ref{E:algo}, converging to $\bar w$.
Then there exists $N \in \N$ such that
\begin{equation*}
(\forall n \geq N) \quad \supp(\bar w) \subset  \supp(w^n) \subset \esupp(\bar w).
\end{equation*}
\end{theorem}

\noindent \textbf{Relation with previous works:} 
When the the groups have dimension 1 (i.e. $H_g \equiv \R$), the problem boils down to a simple $\ell^1$ regression problem.
In that case, exact support recovery in presence of \eqref{QC} is obtained in \cite{BreLor08}, and extended support control in \cite{HalYinZha07}.
When the groups have finite dimension (i.e. $\dim H_g < + \infty$), exact support recovery under \eqref{QC} is known, see for instance \cite{LiaFadPey14}, which contains results for the large class of \textit{partially smooth} functions. 
Still in finite dimension, extended support control has been recently obtained in \cite{FadMalPey17}, where the authors introduce the class of mirror-stratifiable functions.

To our knowledge, Theorems \ref{T:exact recovery with QC} and \ref{T:identification with no QC} are the first to treat the support recovery/control of an algorithm for solving a multiple kernel learning problem involving possibly infinite-dimensional kernels.
Our theoretical results rely on an extension of \cite{FadMalPey17} to the separable Hilbert space setting. 
In particular, we underline that these results build on the concept of stratification, and can thus be extended to the large class of mirror stratifiable regularizers (see section \ref{SS:proofs:mirror strat}).

\section{Numerical experiments}

In this section we illustrate our main results, by considering different variants of \eqref{P:ERM}.
For each experiment, we draw data points $(x_i)_{i=1}^m \subset \R^p$ at random  with i.i.d. entries from a zero-mean standard Gaussian distribution. The Hilbert space of functions we consider is either the 
space of linear functionals or the space generated by combinations of Gaussian  kernels.
Then, we take $\alpha_* \in \R^{m \times G}$ randomly among the vectors verifying  $\vert \supp(\alpha_*) \vert = s$, and take $y = K\alpha_* + \varepsilon$, where $\varepsilon$ is a zero-mean white Gaussian noise with standard deviation $10^{-2}$. 
Here $\supp(\alpha_*)$ must be understood as the group-support of $\alpha_* \in\mathbb{R}^{m\times G} = \R^m \times \ldots \times \R^m$.

We consider two problems, for which we take the following parameters $(m,G,s,\lambda)=(50,20,5,0.2)$.
The first is the \textit{group-lasso} setting, where we chose $p=100$, $H_g \equiv \R^{5}$ and $\Phi_g : \R^{100} \rightarrow H_g$ to be the canonical projection on the $g$-th group of coordinates.
The second is the \textit{Gaussian-Kernel} setting, where we choose $p=2$, and each space $H_g$ to be an infinite-dimensional RKHS associated to the Gaussian kernel
\begin{equation*}
k(x,x') = e^{-\frac{\Vert x_i - x_{i'} \Vert^2}{2\sigma_g^2}} \ ,
\end{equation*}
with $(\sigma_g)_{g=1}^G$ being taken at random in $[10^{-1}, 10^1]$.
For each problem, we draw $200$ instances of $((x_i,y_i)_{i=1}^m,\alpha_*)$, and solve the corresponding \eqref{P:ERM} problem with \ref{E:algo}. The step size for the algorithm is $\tau = 0.8 \Vert K \Vert^{-1}$, and we stop the algorithm after $N=5000$ (resp. $N=50000$) iterations for the group-lasso (resp. Gaussian-kernel).

\smallskip

\noindent \textbf{Description of the results}.
 In Figure \ref{F:histograms}, we show the distribution of $\vert \supp(w^N) \vert$ over the $200$ instances of the algorithm.
We see that for the group-lasso, the method recovers most of the time a support of size $5$-$7$, and that the probability of finding another support decays quickly.
For the Gaussian-kernel,  there are a significant number of instances where  $\vert \supp(w^N) \vert < 5$ and the probability of finding a smaller support decays slowly, 
while the probability of finding a large support decays as quickly as in the linear case.
In Figure \ref{F:algos}, we illustrate the evolution of the support along the iterations of \ref{E:algo}.

\smallskip

\noindent \textbf{Discussion on the results}.
An important point to observe is the fact that the algorithm often fails at recovering exactly the desired support.
This means that qualification conditions which are often assumed in the literature to guarantee the stability of the support cannot be taken for granted in practice.
In the light of Theorem \ref{T:identification with no QC}, the results of Figure \ref{F:histograms} seem to suggest that \ref{E:algo} often fails at recovering the support of $\bar w$.

But it is important to keep in mind that another phenomenon enters into play here: $\bar w$ and $w_*$ does not always share the same support, even for small $\lambda$.
This is particularly clear when we look at the iterates whose support is \textit{smaller} than the one of $w_*$: it would contradict Theorem \ref{T:identification with no QC} if $\supp(w_*) = \supp(\bar w)$.
This fact is well-known in the finite dimensional setting, when no qualification condition holds at $w_*$, see e.g. \cite{FadMalPey17}.
To have a finer understanding of our results, it would be interesting to have a result like \cite[Theorem 3]{FadMalPey17} adapted to our infinite dimensional setting.  

\begin{figure}[h]\centering
\begin{tabular}{@{}c@{\hspace{1mm}}c@{}}
\includegraphics[width=.49\linewidth]{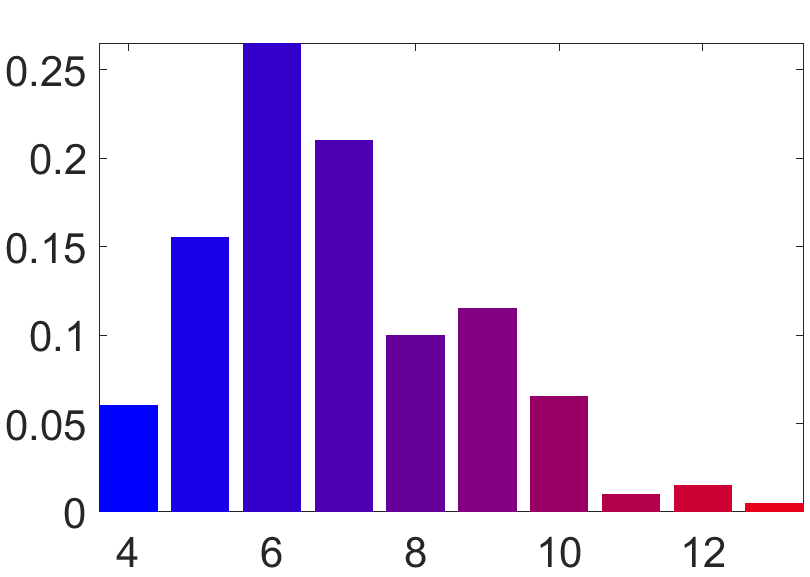} &
\includegraphics[width=.49\linewidth]{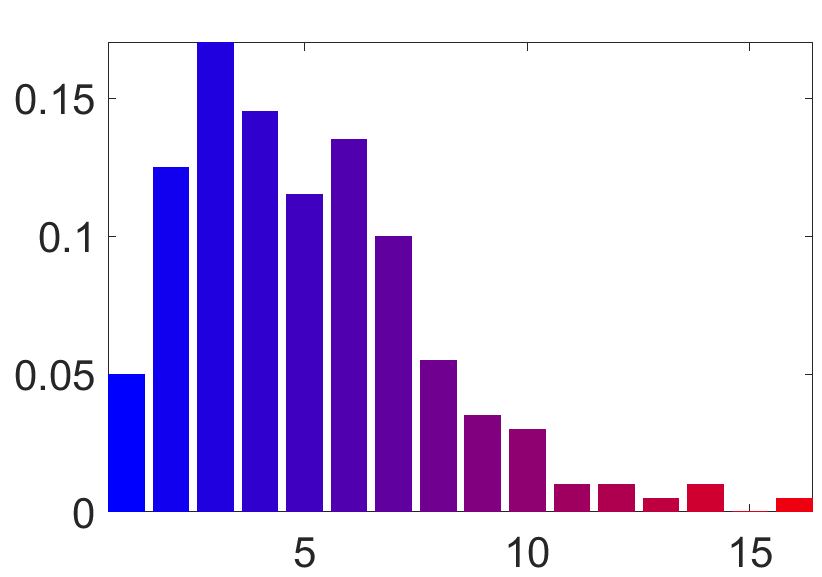}
\end{tabular}
\caption{\label{F:histograms}
Histogram of the size of the support of $w^N$ among 200 realizations of \ref{E:algo}. 
Left: group-lasso. Right: Gaussian-kernel. }
\end{figure}

\begin{figure}[h]\centering
\begin{tabular}{@{}c@{\hspace{1mm}}c@{}}
\includegraphics[width=.49\linewidth]{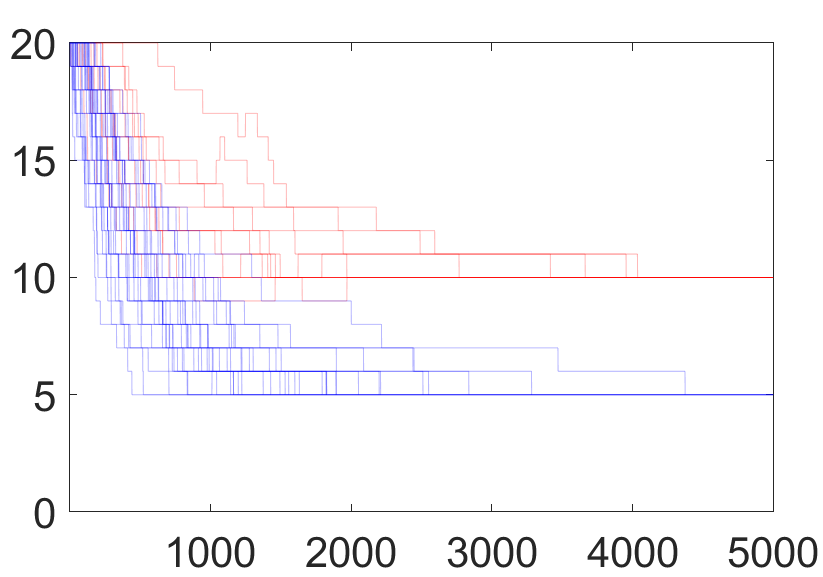} &
\includegraphics[width=.49\linewidth]{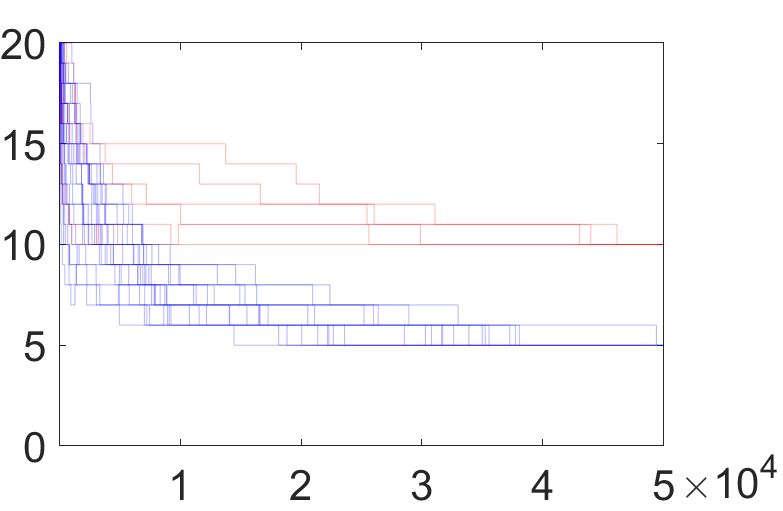}
\end{tabular}
\caption{\label{F:algos}
Evolution of $\supp(w^n)$ while running \ref{E:algo}, for the realizations reaching a support of size $5$ (blue) and $10$ (red). 
Left: group-lasso. Right: Gaussian-kernel. }
\end{figure}

\section{Proofs}

For this section we  need to introduce some notation.
Given $\Omega \subset H$, $\cl \Omega$, $\inte \Omega$ and $\qri \Omega$ will denote respectively the (norm-)closure, interior and quasirelative interior \cite[Definition 6.9]{BauComV2} of $\Omega$.
The closed unit ball, and the corresponding unit sphere of a Hilbert space $H$ are noted $\B_{H_g}$ and $\mathbb{S}_{H_g}$.
The set of all proper lower semi-continuous convex functions on $H$ are denoted $\Gamma_0(H)$.
For $R \in \Gamma_0(H)$, its Fenchel conjugate \cite[Definition 13.1]{BauComV2} is denoted $R^*$.
Finally, for a vector in the product of Hilbert spaces $w \in H = H_1 \times \ldots H_G$, we will often use $\Vert w_g \Vert$ instead of $\Vert w_g \Vert_{H_g}$, if there is no ambiguity.

\subsection{Proof of Proposition \ref{P:CV of the algorithm}}

Let $D : \R^{m \times G} \rightarrow H$ be the linear operator defined by $D \alpha = (X_g^* \alpha_g)_{g \in [\![1.. G ]\!]}$.
Also define $\textbf{y} = (y; \ldots ;y) \in R^{m \times G}$, $\textbf{K} = (K; \ldots ; K) \in \R^{(m\times G)^2}$, and $\widehat{\G}_{\lambda \tau}(\alpha) = (\widehat{\G}_{\lambda \tau,g}(\alpha_g))_{g=1}^G$.
Then,  \ref{E:algo} can be rewritten as
\begin{equation*}
w^{n+1} = D \widehat{\G}_{\lambda \tau} \left( \alpha^n - \tau (\textbf{K} \alpha^n - \textbf{y}) \right).
\end{equation*}
It can also be verified that 
\begin{equation*}
\text{${D \circ \widehat{\G}_{\lambda \tau} = \G_{\lambda \tau} \circ D}$, $D\textbf{K} = X^*XD$ and $D\textbf{y} = X^*y$}
\end{equation*}
holds,
where ${\G}_{\lambda \tau}$ is the group thresholding operator:
\begin{equation*}
(\forall g \in [\![  1..G ]\!]) \ 
 (\G_{\lambda} (\alpha))_g := 
\begin{cases}
 0 & \text{ if } \Vert  \alpha_g \Vert \leq \lambda , \\
 \left(1 - \frac{\lambda}{\Vert \alpha_g \Vert} \right) \alpha_g & \text{ otherwise.}
\end{cases}
\end{equation*}
We see then that the iterates of \ref{E:algo} verify:
\begin{align*}
w^{n+1} & = {\G}_{\lambda \tau} \circ D \left( \alpha^n - \tau (\textbf{K} \alpha^n - \textbf{y}) \right), \\
& = {\G}_{\lambda \tau} \left( w^n - \tau X^*(X w^n- y) \right).
\end{align*}
We recognize here the iterates of the Forward-Backward algorithm applied to $R+h$, where
\begin{equation}\label{E:proof def R and h}
R(w):= \sum_{g=1}^G \Vert w_g \Vert, \ h(w):= \Vert Xw - y \Vert^2/(2 \lambda).
\end{equation}
Since $R$ is coercive and $h$ is bounded from below, we know that \eqref{P:ERM} admits solutions, so we deduce from \cite[Corollary 28.9]{BauComV2} that the sequence converges weakly towards a solution of  \eqref{P:ERM}.
Since $(w^n)_\nin$ belongs to $\im D$, which has finite dimension, this weak convergence actually turns out to be strong.

\subsection{Identification for mirror-stratifiable functions}\label{SS:proofs:mirror strat}

To prove Theorems \ref{T:exact recovery with QC} and \ref{T:identification with no QC}, we adapt to  our Hilbert setting the tools of stratification (of a set) and mirror-stratifiablilty (of a function) introduced in \cite{FadMalPey17}.

We say that $\M = \{M_i\}_{i \in I} \subset 2^H$ is a \textit{stratification} of a set $\Omega \subset H$, if it is a finite partition of $\Omega$ verifying, for any pair of strata $M$ and $M'$ in $\M$:
\begin{equation}\label{E:strata def}
M \cap \cl M' \neq \emptyset \Rightarrow M \subset \cl M'.
\end{equation}
Since it is a partition, we note $M_w \in \M$ the unique strata containing a vector $w \in \Omega$.
It is important to note that the property \eqref{E:strata def} endows $\M$ with a partial order $\preceq$, given by:
\begin{equation*}
M \preceq M' \Leftrightarrow M \subset \cl M'.
\end{equation*}
We are now in position to define the mirror-stratifiability of a convex function $R \in \Gamma_0(H)$.
For this, we will need to associate to $R$ a correspondance operator between subsets of $\dom \partial R$ and $\dom \partial R^*$:
\begin{equation*}
(\forall S \subset \dom \partial R) \quad J_R(S) := \bigcup\limits_{w \in S} \qri \partial R(w).
\end{equation*}

\begin{definition}
Let $R \in \Gamma_0(H)$, and let $\M$ (resp. $\M^*$) be a stratification of $\dom \partial R$ (resp. $\dom \partial R^*$).
We say that $R$ is mirror-stratifiable (w.r.t. $\M$ and $\M^*$) if $J_R$ realizes a bijection between $\M$ and $\M^*$ (with inverse $J_{R^*}$), and  is decreasing:
\begin{equation*}
(\forall (M,M') \in \M^2) \quad  M \preceq M' \Leftrightarrow J_R(M') \preceq J_R(M).
\end{equation*}
\end{definition}

Our main result for mirror-stratifiable functions controls the stability of their stratas up to perturbations on the graph of $\partial R$:

\begin{proposition}\label{T:mirror strat id}
Let $R \in \Gamma_0(H)$ be mirror-stratifiable w.r.t. $(\M,\M^*)$, and  consider a converging sequence on the graph of $\partial R$:
\begin{equation*}
\eta^n \in \partial R(w^n), \quad w^n \underset{n \to + \infty}{\longlongrightarrow} \bar w, \quad \eta^n \underset{n \to + \infty}{\longlongrightarrow} \bar \eta.
\end{equation*}
Then there exists $N \in \N$ such that
\begin{equation*}
(\forall n \geq N) \quad M_{\bar w} \preceq M_{w^n} \preceq  J_{R^*}(M^*_{\bar \eta}).
\end{equation*}
\end{proposition}

\begin{proof}
Reasoning exactly as in \cite[Proposition 1]{FadMalPey17}, we obtain the existence of some $\delta>0$ such that 
\begin{equation*}
\max \{\Vert w - \bar w \Vert , \Vert \eta - \bar \eta \Vert \} \leq \delta \Rightarrow M_{\bar w} \preceq M_w \text{ and } M_{\bar \eta}^* \preceq M_\eta^*.
\end{equation*}
From now, we consider that $n$ is large enough so that $\max \{\Vert w^n - \bar w \Vert , \Vert \eta^n - \bar \eta \Vert \} \leq \delta$ holds.
Thus, the first inequality $M_{\bar w} \preceq M_{w^n}$ holds.
Using \cite[Fact 6.14.ii]{BauComV2}, we see that
\begin{equation*}
\eta^n \in \partial R(w^n) = \cl \qri \partial R(w^n) = \cl J_{R}(\{w^n\}) \subset \cl J_R(M_{w^n}),
\end{equation*}
or, equivalently,  $M^*_{\eta^n} \preceq J_R(M_{w^n})$.
Considering that $M_{\bar \eta}^* \preceq M_{\eta^n}^*$, we deduce from the transitivity of $\preceq$ that $M_{\bar \eta}^* \preceq J_R(M_{w^n})$.
Since we assume $J_R$ to be bijective and decreasing, the latter is equivalent to
$M_{w^n} \preceq J_{R^*}(M^*_{\bar \eta})$, which ends the proof.
\end{proof}

\subsection{Proof of Theorem \ref{T:identification with no QC}}

Let $R$ and $h$ be as in \eqref{E:proof def R and h}, and let us define 
\begin{align*}
\M := & \left\{ \prod_{g=1}^G M_g : (\forall g \in [\![ 1..G ]\!]) \ M_g \in \{ H_g \setminus \{0\}, \{0 \} \} \right\},\\
\M^* := & \left\{ \prod_{g=1}^G M_g^* : (\forall g \in [\![ 1..G ]\!]) \ M_g^* \in \{  \mathbb{S}_{H_g} , \inte \B_{H_g}\} \right\}.
\end{align*}
It is a simple exercise to verify that $\M$ is a stratification of $\dom \partial R = H$, and that $\M^*$ is a stratification of $\dom \partial R^* = \prod_g \B_{H_g}$.
Then, note that for all $w \in H$ and $\eta \in \prod_g \B_{H_g}$:
\begin{equation*}
\partial R(w) = \prod_g M^*_g, \text{ with } M^*_g =
\begin{cases}
 \frac{w_g}{\Vert w_g \Vert} & \text{if } w_g \neq 0 \\
 \inte \mathbb{B}_{H_g} & \text{else,}
\end{cases}
\end{equation*}
\begin{equation*}
\partial R^*(\eta) = \prod_g M_g, \text{ with } M_g =
\begin{cases}
 [0,+\infty[ \eta_g & \text{if } \eta_g \in \mathbb{S}_{H_g} \\
 \{0\} & \text{else.}
\end{cases}
\end{equation*}
Using \cite[Proposition 2.5]{BorLew92}, it is  easy to deduce that $J_R$ induces a bijection between and $\M$ and $\M^*$, by putting in relation $H_g \setminus \{0 \}$ with $\mathbb{S}_{H_g}$, and $\{0\}$ with $\inte \B_{H_g}$.
It can be shown by the same arguments that $J_R$ is decreasing for $\preceq$, meaning that $R$ is mirror-stratifiable w.r.t. $(\M,\M^*)$.

Now, according to the proof of Proposition \ref{P:CV of the algorithm}, the sequence $(w^n)_\nin$ verifies
\begin{equation*}
w^{n+1} = \G_{\lambda\tau}(w^n - \tau \lambda \nabla h(w^n)),
\end{equation*}
which can be equivalently rewritten as
\begin{equation*}
\eta^n:= \frac{w^n - w^{n+1}}{\lambda \tau} - \nabla h(w^n) \in \partial R(w^{n+1}).
\end{equation*}
We know from Proposition \ref{P:CV of the algorithm} that $w^{n+1}$ converges strongly to $\bar w$, so we can deduce from the continuity of $\nabla h$ that $\eta^n$ converges to $\bar \eta := -\nabla h(\bar w)$.
Applying Proposition \ref{T:mirror strat id}, we deduce that for $n$ large enough,
\begin{equation*}
M_{\bar w} \preceq M_{w^n} \preceq J_{R^*}(M^*_{\bar \eta}).
\end{equation*}
From the definition of $\M$, we see that $M_{\bar w} \preceq M_{w^n}$ is equivalent to $\supp(\bar w) \subset \supp(w^n)$.
Also, we have from the definition of $J_{R^*}$ and $\M^*$ that
\begin{equation*}
J_{R^*}(M^*_{\bar \eta}) = \{ w \in H : w_g \neq 0 \text{ if } \Vert \bar \eta_g \Vert =1, w_g =0 \text{ else} \}.
\end{equation*}
According to the definition of $\esupp(\bar w)$, we deduce that $ M_{w^n} \preceq J_{R^*}(M^*_{\bar \eta})$ is equivalent to $\supp(w^n) \subset \esupp(\bar w)$, which concludes the proof.

\subsection{Proof of Theorem \ref{T:exact recovery with QC}}

Theorem \ref{T:exact recovery with QC} follows directly from Theorem \ref{T:identification with no QC}, by  observing that \eqref{QC} is equivalent to $\supp(\bar w) = \esupp (\bar w)$.

\end{document}